\documentclass[a4paper,11pt]{article}%{report} 
\usepackage{geometry}
\usepackage[affil-it]{authblk}
\usepackage{amsthm,amsmath,amssymb,amsfonts}
\usepackage{microtype,mathrsfs}
\usepackage{epsfig,color}
\usepackage{mathrsfs}
\usepackage{dsfont}
\usepackage{hyperref}
\usepackage{enumerate}
\usepackage{enumitem}
\usepackage{comment}
\setlist[enumerate]{leftmargin=*}

\let\OLDthebibliography\thebibliography
\renewcommand\thebibliography[1]{
  \OLDthebibliography{#1}
  \setlength{\parskip}{3pt}
  \setlength{\itemsep}{0pt plus 0.3ex}
}

\theoremstyle{plain}
\newtheorem{thm}{Theorem}[section]
\newtheorem{defn}[thm]{Definition}
\newtheorem{prop}[thm]{Proposition}
\newtheorem{cor}[thm]{Corollary}
\newtheorem{rmk}[thm]{Remark}
\newtheorem{lma}[thm]{Lemma}

\newtheorem{exm}[thm]{Example}

\newcommand{\eq}{\begin{equation}}
\newcommand{\qe}{\end{equation}}

\def\N{{\rm I\kern-0.16em N}}
\def\R{{\rm I\kern-0.16em R}}
\def\E{{\rm I\kern-0.16em E}}
\def\P{{\rm I\kern-0.16em P}}
\def\F{{\rm I\kern-0.16em F}}
\def\B{{\rm I\kern-0.16em B}}
\def\Z{{\rm I\kern-0.46em Z}}
\def\C{{\rm I\kern-0.46em C}}
\def\G{{\rm I\kern-0.50em G}}

\newcommand{\mF}{\mathcal F}
\newcommand{\indic}[1]{\mathbf 1_{#1}}

\title{An algebra of Stein operators}
\author{Robert E. Gaunt\footnote{School of Mathematics, The University of Manchester, Manchester M13 9PL, UK \texttt{robert.gaunt@manchester.ac.uk} }, \
Guillaume Mijoule 
\footnote{MOKAPLAN, INRIA de Paris, 2 rue Simone Iff, CS 42112, 75589 Paris Cedex 12, France  \texttt{guillaume.mijoule@inria.fr}. } \  and Yvik Swan\footnote{Universit\'e de Li\`ege, Sart-Tilman, All\'ee de la d\'ecouverte
  12, B-8000 Li\`ege, Belgium  \texttt{yswan@uliege.be}. }} 
\affil{The University of Manchester, INRIA and Universit\'e de Li\`ege}
\date{}
\begin{document}
\maketitle

\begin{abstract}
  We build upon recent advances on the distributional aspect of Stein's method to propose a novel and flexible technique for computing Stein operators for random variables that can be written as products of independent random variables. We show that our results are valid for a wide class of distributions including normal, beta, variance-gamma, generalized gamma and many more.  Our operators are $k$th degree differential operators with polynomial coefficients; they are straightforward to obtain even when the target density bears no explicit handle.  As an application, we derive a new formula for the density of the product of $k$ independent symmetric variance-gamma distributed random variables.   

  \


% 60B10

%26D10

\

\end{abstract}

%\tableofcontents

%\section{Introduction} 

%\subsection{Stein's method}

\section{Introduction} 

%\subsection{Stein's method}

In 1972, Charles Stein (1920--2016) \cite{stein} introduced a powerful
method for estimating the error in normal approximations.  The method
was adapted to the Poisson distribution by Louis Chen in \cite{chen
  0}, and has since been extended to a very broad family of
probability distributions.  The general procedure for a given target
distribution $p$ is as follows.  In the first step, one obtains a
suitable operator $A$ acting on a class of test functions
$\mathcal{F}$ such that $\E[Af(X)]=0$ for all $f\in \mathcal{F}$; the
operator $A$ is called a \emph{Stein operator} for $p$.  For
continuous distributions, $A$ is a typically a differential operator;
for the standard normal distribution, the classical operator is
$Af(x)=f'(x)-xf(x)$.  One then considers the so-called \emph{Stein
  equation}
\begin{equation}\label{seqn}Af_h(x)=h(x)-\E h(X),
\end{equation}
where $h$ is a real-valued test function.  If $A$ is well chosen then,
for a given $h$, the Stein equation (\ref{seqn}) can be solved for
$f_h$.  The second step of the method consists of obtaining this
solution and then bounding appropriate lower order derivatives.
Evaluating both sides of (\ref{seqn}) at a random variable of interest
$W$ and taking the supremum over all $h$ in some class of functions
$\mathcal{H}$ leads to the estimate
\begin{equation}\label{deqn}d_\mathcal{H}(\mathcal{L}(W),\mathcal{L}(X)):=\sup_{h\in\mathcal{H}}|\E
  h(W)-\E h(X)|\leq \sup_{f_h}|\E[Af_h(W)]|,
\end{equation}
where the final supremum is taken over all $f_h$ that solve
(\ref{seqn}).  The third and final step of the method involves
developing appropriate strategies for bounding the expectation on the
right hand side of (\ref{deqn}). This is of interest because many
important probability metrics (such as the Kolmogorov and Wasserstein
metrics) are of the form
$d_\mathcal{H}(\mathcal{L}(W),\mathcal{L}(X))$.  Moreover, in many
settings bounding the expectation $\E[Af_h(W)]$ is relatively
tractable, and as a result Stein's method has found application in
disciplines as diverse as random graph theory \cite{bhj92}, number
theory \cite{harper}, statistical mechanics \cite{el10} and quantum
mechanics \cite{mps17}.  We refer to the survey paper \cite{ross} as
well as to the monographs \cite{nourdinpec2011, chgolshao2012} for a
deeper look into some of the fruits of Charles Stein's seminal
insights, particularly in the case where the target is the normal
distribution.

%\subsection{Stein operators}  

% In this paper, we introduce general techniques and theory for
% obtaining the first ingredient of Stein's method.  
%In the sequel we adopt the following lax definition:

The linchpin of the method is the operator $A$ whose properties are
crucial to the success of the whole enterprise. %  and, clearly, the whole
% enterprise is doomed if $A$ is not well chosen
 In the sequel, we
concentrate exclusively on \emph{differential} Stein operators (some
operators in the literature are integral or even fractional, see e.g.\
\cite{Liu,ah17}) and adopt the following lax definition:

\begin{defn}\label{def:weaksteinop}
  A linear differential operator $A$ acting on a class $\mathcal{F}$
  of functions is a \emph{Stein operator} for $X$ if (i) $Af \in L^1(X)$
 and (ii) $\E \left[ Af(X) \right]=0$ for all
  $f \in \mathcal{F}$.  
 \end{defn}
 
 There are infinitely many Stein operators for any
 given target distribution. For instance, if the distribution is known
 (even if only up to a normalizing constant) then the ``canonical''
 theory from \cite{ley} applies, leading to entire {families} of
 operators. This approach provides natural first order polynomial
 operators e.g.\ for target distributions which belong to the Pearson
 family \cite{schoutens} or which satisfy a diffusive assumption
 \cite{dobler beta,kusuotud}.  In some cases, one may rather apply a
 {duality} argument. For instance the p.d.f$.$
 $\gamma(x) = (2\pi)^{-1/2}\mathrm{e}^{-x^2/2}$ of the standard normal
 distribution satisfies the first order ODE
 $\gamma'(x) + x \gamma(x)=0$ leading, by integration by parts, to the
 already mentioned operator $ Af(x) = f'(x) -xf(x)$. This is
 particularly useful for densities defined implicitely via ODEs.  % Similarly the
%  exponential distribution has natural operator
% \begin{equation}
%   \label{eq:12}
%    Af(x) = xf'(x) +(1-x)f(x)
% \end{equation}
% with $f \in \mathcal{F}$ the class of all differentiable functions
% such that both $xf'(x)$ and $(1-x)f(x)$ are integrable with respect to
% the normal measure. 
 Such are by no means the only methods for deriving differential Stein
 operators and, for any given $X$, one can easily determine an entire
 ecosystem of Stein operators, leading to the natural question of
 \emph{which operator to choose}.  % A general canonical theory of
 % Stein operators is available in \cite{ley}, and many other general
 % theories have been proposed in recent years, see
 % https://sites.google.com/site/steinsmethod/ for an overview of the
 % quite large literature on this topic.  The first key to setting up
 % Stein's method for a target $X$ is to identify a ``good'' operator,
 % i.e.\ an operator $A$ whose properties are relevant to the problem at
 % hand.
 One natural way to sieve through the available options is to further
 impose that the chosen operator be {characterizing} for $X$,
 i.e.\ that if some $Y$ enjoys the property that $\E[f(Y)]=0$ for all
 $f \in \mathcal{F}$, then $Y {=} X$ (equality in law). Such
 requirements often do not suffice and will not be imposed here; our
 focus will rather be on another crucial quality of a ``good'' Stein
 operator: tractability. More precisely, we will focus solely on Stein
 operators which satisfy the next definition.
 \begin{defn}\label{def:polyopera}
   We call a Stein operator \emph{polynomial} if it can be written as
   a finite sum $A = \sum_{i,j} a_{ij} M^iD^j $ for real coefficients
   $a_{ij} \in \R$, with $M(f) = \left(x\mapsto xf(x) \right)$ and
   $D(f) = \left(x\mapsto f'(x) \right)$.
 \end{defn}
 
 Except in the most basic cases, determining polynomial Stein
 operators is not an easy task. Interestingly, many densities do not
 admit a first order polynomial Stein operator and it is necessary to
 consider higher order operators: \cite{gaunt vg} obtains a second
 order operator for the entire family of variance-gamma distributions
 (see also \cite{eichtha14} and \cite{gaunt gh}), \cite{pike} obtain a
 second order Stein operator for the Laplace distribution, and
 \cite{pekoz} obtain a second order operator for the PRR distribution,
 which has a density that can be expressed in terms of the Kummer $U$
 function.  % See also \cite[Examples 4 and 42]{ley} where the same
 % operators are obtained through direct ODE arguments; as one can see
 % such brute force calculations can rapidly lead to intractable
 % computations. 
 If the p.d.f$.$ of $X$ is defined in terms of special
 functions (Kummer $U$, Meijer $G$, Bessel, etc.)  which are
 themselves defined as solutions to explicit $d$th order differential
 equations then the duality approach shall yield a tractable
 differential operator with explicit coefficients.

 In many cases, the target distribution is not even defined
 analytically in terms of its distribution but rather
 probabilistically, as a statistic (sum, product, quotient) of
 independent contributions.  Explicit knowledge of the density of such
 random variables is then generally unavailable and, in order to
 obtain polynomial Stein operators for such objects,
% is generally out of reach, and
new approaches must be devised. In \cite{aaps16,aaps18}, a
Fourier-based approach is developed for identifying appropriate
operators for arbitrary combinations of independent chi-square
distributed random variables.  In \cite{gaunt ngb,gaunt pn}, an
iterative conditioning argument is provided for obtaining operators
for (mixed) products of independent random beta, gamma and mean-zero
normal random variables. 
In this context we are naturally lead to the following research
problem:
\begin{quote}
\emph{Given independent random variables
  $X_1, \ldots, X_d$ with polynomial operators $A_1, \ldots, A_d$,
  respectively, can one deduce a tractable polynomial Stein operator
  for statistics of the form $X = F(X_1, \ldots, X_d)$?
}
\end{quote} 
In this paper, we provide an
answer for functionals of the form
$F(x_1, \ldots, x_d) = x_1^{\alpha_1}\cdots x_d^{\alpha_d}$ with
$\alpha_i \in \R$ and $X_i$'s with polynomial Stein operator
satisfying a specific commutativity assumption (Assumption~3 below). 
% A general answer to this question for arbitrary functionals of
% arbitrary targets is out of reach of the current state-of-the-art on
% Stein's method.  In fact, even in the case of products of normal
% random variables there are great difficulties and important questions
% remain -- to this day -- unsolved.  

This paper is mostly devoted to the problem of obtaining Stein
operators, which allows for a focused treatment of their theory.  We
acknowledge, though, that further work is required before the Stein
operators obtained in this paper can be used to prove approximation
theorems via Stein's method.  However, the theory of Stein operators
to which this paper contributes is of importance in its own
right. Indeed, there are now a wide variety of techniques which allow
one to obtain useful bounds on solutions to the resulting Stein
equations (see, for example, \cite{kumar16,dgv,ah17}) and which can be
adapted to the operators that we derive.  Also, and this has now been
demonstrated in several papers such as
\cite{nourdinpecswan2013,nourdinpecswan2014,aaps16,amps16,aaps18},
Stein operators can be used for comparison of probability
distributions directly without the need of solving Stein equations;
such an area is also the object of much interest.  Finally, we stress
that Stein operators are also of use in applications beyond proving
approximation theorems; for example, in obtaining distributional
properties \cite{gaunt pn, gaunt ngb} and other surprising
applications include the derivation of formulas for definite integrals
of special functions \cite{gaunt ram}.  Indeed, in Section
\ref{sec:form-dens-char}, we propose a novel general technique for
obtaining formulas of densities of distributions that may be
intractable through other existing methods.

The outline of the paper is as follows. In Section
\ref{sec:general-results} we identify the key elements allowing to
construct a form of ``operator algebra'' which provides -- by
elementary calculations -- polynomial operators for $X$'s which can be
written as products (see Section \ref{sec:prod-distr}) and powers (see
Section \ref{sec:powers-inverse-distr}) of independent
contributions. We apply the theory in Section \ref{sec:applications-1}
to recover several operators from contemporary literature on Stein's
method and also to provide many new ones. Finally, in Section
\ref{sec:form-dens-char} we consider an application of operators obtained by our method
to finding densities of product distributions.

\section{An algebra of Stein operators}
\label{sec:general-results}

\subsection{About the class of functions on which the operators are defined}

\label{sec:classfunctions}

% We first briefly discuss the class of functions $\mathcal F$ on which
% the Stein operators are defined. 

All random variables we consider in the paper satisfy the following
assumption: 

\

\noindent \textbf{Assumption 1:} $X$ admits a smooth density $p$ with respect to the
Lebesgue measure on $\R$; this density is defined and non-vanishing on
some (possibly unbounded) interval $J \subseteq \R$. 

\

By definition, a Stein operator $A$ for a random variable $X$ acts on
a collection $\mathcal{F}$ of functions for which the expectations
vanish. Although determining the largest possible set $\mathcal F$ may
be an interesting quest, it will not be part of ours because this can
only be done on a case-by-case basis % (and, even in the well-trodden
% case of the normal distribution with operator $Af(x) = f'(x) - xf(x)$,
% there is no simple characterisation of this largest set). 
and the focus of our paper is the construction of an algebra allowing
to generate tractable operators.  In order to ensure that the
operators we obtain do not act on trivial classes of functions (e.g.\
$\mathcal{F} = \left\{ \emptyset \right\}$), we shall
simply % impose throughout the paper that our operators act on
% a ``large'' class $\mathcal F$, by
impose the following assumption: 

\

\noindent \textbf{Assumption 2:}  $X$   admits an
operator $A$ acting on $\mathcal F$ which contains the set of smooth
functions with compact support $\mathcal C_0^\infty(\R)$.  

\

Assumption~2 is not too restrictive in our context (see Remark
\ref{rmk:about-class-funct} below), although we will need to reinforce
it slightly in Section \ref{sec:powers-inverse-distr}.  A collateral
benefit of restricting to random variables satisfying Assumptions 1
and~2 is that we now may consider samples $X_1, \ldots, X_n$ of random
variables with respective operators $A_1, \ldots, A_n$ acting on their
respective classes $\mathcal{F}_1, \ldots, \mathcal{F}_n$ and we are
ensured that $\bigcap_{i=1}^n\mathcal{F}_i \supseteq C_0^{\infty}(\R)$.
%(or $C_0^{\infty}(\R^+)$). 
Consequentially, it is guaranteed that any
statements on the joint behaviour of any of the $A_i$ will also hold on
non-trivial classes of functions.
\begin{rmk}
  In some cases (for instance when $X$ has exponential moments), one
  can easily extend $\mathcal{F}$ to smooth functions $f$ such that
  $f^{(k)}$ has at most polynomial growth for all $k\geq 0$ (and in
  particular, to polynomials). 
  %such extensions are left to the interested reader on a case by case basis. 
\end{rmk}

\begin{rmk} \label{rmk:about-class-funct} 
Under Assumption~1, the ``canonical" Stein operator in \cite{ley} is
  $A_c \;: \; f \mapsto{(fp)'/{p}} = f' + ({p'}/{p}) f$ and we
  have $\E[A_c f (X)] = 0$ at least for those functions $f$ such that
  $f(x) p(x) \rightarrow 0$ on the border of $J$.  If
  $J = (-\infty,+\infty)$, it is clear that any
  $f \in \mathcal C_0^\infty(\R)$ satisfies $f(x) p(x) \rightarrow 0$
  on the border of $J$, hence Assumption 2 is automatically
  satisfied. On the other hand, if $J$ has a finite border, say
  $J= (a,+\infty)$ with $a\in\R$, then $f \in \mathcal C_0^\infty(\R)$
  does not imply necessarily that $f p$ vanishes on the border (see
  \cite{chatt} for the case of exponential approximation). An easy
  workaround in this case is to apply the operator to $(x-a) f$
  instead of $f$, which leads to the new operator $A = A_c
  (M-aI)$. This operator satisfies $\E[Af(X)] = 0$ for all
  $f \in \mathcal C_0^\infty(\R)$. Of course, this has to be done on
  both borders of the support if they are both are finite.
\end{rmk}

 %For instance, the canonical Stein operator for the Beta
% distribution is 
% $A_c f = f' + \left( \frac{\alpha-1}{x} - \frac{\beta-1}{1-x} \right)
% f$, but we will use instead the operator
% $A_c M(I-M) f = x(1-x)f' - [(\alpha+\beta)x-\alpha] f$ (which is
% actually the one commonly used in the Stein literature; see
% \cite{dobler beta}, \cite{goldstein4}).

% In some cases, we will also have to use different operators that the
% ones commonly used in the Stein literature, in order for them to
% satisfy Assumption~1. For instance, the standard Stein operator for
% the Gaussian distribution is $A_c f = f'-xf$, but we will use
% instead a right-multiplication by $M$ of it,
% $A f=A_c M f = xf' + (1-x^2) f = (T_1 - M^2) f$.  One could argue
% that by doing so, we restrict the standard operator to those
% functions $f$ such that $f(0)=0$. But the artificial
% right-multiplication by $M$ we applied on the underlying Stein
% operator, is still present in the operator for a product
% distribution: for instance, we will show that a Stein operator for
% the product of two Gaussian distributions is $T_1^2 - M^2$, which
% can be rewritten $DMDM-M^2$, and we can 'divide' by $M$ on the right
% (or equivalently, apply the operator to $f/x$ instead of $f$), to
% obtain the operator $DMD-M$ (which is the one of \cite{gaunt
% vg}). The last operator acts (at least) on
% $\mathcal C_0^\infty(\R)$.

\subsection{The building blocks}
\label{sec:notations}

Let us record some notation regarding the different operators that
will be used throughout the paper.  We let $\mathcal F \supseteq \mathcal
C_0^{\infty}(\R)$; $M$ is the multiplication operator: $M(f) =
\left(x\mapsto xf(x)  \right)$; $D$ the differentiation operator $D(f)
= f'$; $I$ the identity of $\mathcal F$; for $a \in \R\setminus\{0\}$,
$\tau_a (f) = \left( x \mapsto f(ax) \right)$; and $\forall r \in \R,
T_r = MD + r I$.  

%$\mathcal F$ is a space of smooth functions, stable under multiplication and differentiation. We assume that $\mathcal F$ satisfies Constraint \ref{item:1}.

 \begin{rmk}\label{rmk:infiniteparam}
  We will also need to consider the limit of operator $T_r$ as
  $r\to\infty$. Although such a limit is badly defined, we note how
  $\lim_{r\to\infty} r^{-1} T_r = I$ (pointwisely for any $f \in \mathcal F$). % Identities \eqref{eq:recT} also hold, after the appropriate rescaling, as
% $r\to\infty$.  
\end{rmk}

Our starting point is
the following extension of one of the main results of \cite{gaunt
  ngb}:

\begin{prop}
\label{prop:1}
Let $\mathcal F \supseteq \mathcal C_0^{\infty}(\R)$.  Assume $X,Y$ are
random variables with respective Stein operators
\begin{eqnarray}
A_X &=& L_X - M^p K_X,\label{eq:1}\\
A_Y &=& L_Y - M^p K_Y, \label{eq:2}
\end{eqnarray}
where $p\in \N$ and where the operators $L_X,K_X,L_Y,K_Y$ commute with
each other and with every $\tau_a$, $a \in \R$. Then, if $X$ and $Y$
are independent, \eq L_XL_Y - M^p K_XK_Y
\qe
is  a Stein operator for $XY$. 

\end{prop}
\begin{proof}
Let $f \in \mF$. Using a conditioning argument and the commutative property between the different operators, we have that
\begin{align*}
\E[L_XL_Y f (XY)] & = \E [\E[ \tau_Y L_X L_Y f (X)\, | Y \,] ] \\
& = \E [\E[  L_X \tau_Y L_Y f(X)\, | Y \,] ]\\
& = \E [\E[  M^p K_X \tau_Y L_Y f(X)\, | Y \,] ]\\
&=  \E [\E[  X^p \tau_Y K_X  L_Y f (X)\, | Y \,] ]\\
&=  \E [X^p  K_X  L_Yf (XY) ]\\
&=  \E [X^p  \E[\tau_X K_X  L_Y f(Y)\, | \,X ] ] \\
&= \E [X^p  \E[L_Y  \tau_X K_Xf (Y)\, | \, X] ]\\
&=\E [X^p  \E[M^p  K_Y\tau_X K_Xf (Y)\, | \, X] ]\\
&=\E [X^p Y^p   K_Y K_Xf (XY)\, ],
\end{align*}
which achieves the proof.
\end{proof}

The assumption that the operators commute with scaling $\tau_a$
is crucial for the proof of Proposition \ref{prop:1}; it will also
reveal itself to be the linchpin of our ``operator algebra''.
Restricting to first order differential operators we deduce that the fundamental
operators $L_X,K_X,L_Y,K_Y$ need to be of the form 
$ T_r = MD + r I,$
as these are the only first
order polynomial Stein operators which commute with the multiplication
operator (at least for a non trivial  class $\mathcal F$). A
fundamental subalgebra of linear operators, which will play a
prominent role in this work, is the algebra $\mathcal T$ composed
  of all linear combinations and compositions of 
  $T_r$'s for $r \in \R$. 
Note also that $\mathcal T$ is the set of operators that are
polynomials of the operator $MD$. Since each $T_r$ commutes which each
$\tau_a$, so does any element of $\mathcal T$.  These considerations
naturally lead to the following assumption which will underpin the
entire theory we develop:

\

\noindent \textbf{Assumption~3:} There exist $k\in \N$ and linear
operators $L, K$ such that $X$ admits a Stein operator (in the sense
of Definition \ref{def:weaksteinop}) of the form
\begin{equation}
  \label{eq:29}
  A = L - M^k K,
\end{equation}
where the operators $L$, $K$ are elements of $\mathcal T$.

\

In most situations  that we consider, however, the operators $K$ and $L$ will be products of $T_r$ operators.  We now collect some useful relations for the operators that will be used throughout this paper.

\begin{lma}Let $r,r'\in\R$, $a\in \R\setminus\{0\}$ and $n\in\N$. Then, acting on the class of functions $\mathcal F \supseteq \mathcal C_0^{\infty}(\R)$, the operators $(M,D,T_r,\tau_a)$ satisfy the following relations:
\[\tau_a M = a \, M \tau_a  \quad  \mbox{and} \quad
D \tau_a = a \, \tau_a D, 
\]
and 
\begin{equation}
\label{eq:recT}T_r M^n = M^n T_{r+n} \quad  \mbox{and} \quad T_r D^n = D^n T_{r-n}.
\end{equation}
Additionally, $T_r$ and $T_{r'}$ always commute, and every $T_r$ commutes with every $\tau_a$.
\end{lma}

\begin{comment}
Using the fact that $DM = MD +I$, one can easily check that $\forall r \in \R\cup \{\infty\}, \forall n \in \N$,
\eq
\label{eq:recT}
 T_r M^n = M^n T_{r+n},
\qe
and
\eq
T_r D^n = D^n T_{r-n},
\qe
with the usual convention that $r+\infty = \infty$. It is also direct to see that
$$\tau_a M = a \, M \tau_a,$$
and
$$D \tau_a = a \, \tau_a D.$$
Note also that $T_r$ and $T_{r'}$ always commute (since they are polynomials, of degree 1, in $MD$), and that every $T_r$ commutes with every $\tau_a$.
\end{comment}

\begin{proof}(i) Here, and throughout the proof, let $f\in \mathcal F$.  Then, $\tau_aMf(x)=\tau_axf(x)=axf(ax)=aM\tau_af(x)$ and $D\tau_af(x)=Df(ax)=af'(ax)=a\tau_aDf(x)$, as required.

(ii) By the product rule of differentiation, one has that $DM=MD+I$.  Therefore, $T_rM=MDM+rM=M^2D+(r+1)M=MT_{r+1}$, and the first relation now follows from direct recurrence. Similarly, we have 
$T_rD=rMD^2+D=rDMD+(r-1)D=DT_{r-1},$
and the second relation now follows from direct recurrence.

(iii) That $T_r$ and $T_{r'}$ commute follows since they are polynomials, of degree 1, in $MD$.  Also, $T_r\tau_af(x)=T_rf(ax)=axf'(ax)+rf(ax)=\tau_a(xf'(x)+rf(x))=\tau_aT_rf(x)$, and therefore $T_r$ and $\tau_a$ commute.
\end{proof}

Note also that if we define, for an operator $L$, $L\mathcal T: \{ L A \; ; \; A \in \mathcal T \}$, and similarly $\mathcal T L$, then a direct consequence of \eqref{eq:recT} is that for any $n \in \mathbb Z_+$,
\eq
\label{eq:commT}
M^n \mathcal T = \mathcal T M^n \quad  \mbox{and} \quad D^n \mathcal T = \mathcal T D^n.
\qe

\subsection{More on Assumption~3}

The purpose of this section is to give a simple criterion to identify
operators $A$ that satisfy Assumption~3. We have the following
criterion for $X$ to satisfy Assumption~3, when its Stein operator is
written in an expanded form.
\begin{lma}\label{lma:building-blocks}
$X$ satisfies Assumption~3 if, and only if, $X$ has a Stein operator of the form $ \sum_{i, j}a_{ij}M^jD^i$ with
  $\# \left\{ j-i\, | \, a_{ij}\neq0 \right\}\le 2$.
\end{lma}
\begin{proof}
Let us give a preliminary result. Let $L = M^{k_1} D^{l_1} \ldots M^{k_n} D^{l_n}$. Let $q = \sum_{i}k_i - \sum_{i} l_i$. Then if $q\geq0$, $L \in M^q \mathcal T$. Indeed, note that $MD = T_0$ and $DM = T_1$. Hence, in the product $M^{k_1} D^{l_1} \ldots M^{k_n} D^{l_n}$, one can pair any product $MD$ (or $DM$), replace it by $T_0$ (or $T_1$), and flush it right using \eqref{eq:recT}. Hence the result.

  Now we prove the Lemma. If $X$ satisfies Assumption~3, then a Stein operator for $X$ is $A=P(MD) - M^k Q(MD)$, where $P$ and $Q$ are polynomials. Using repeatedly the fact that $DM = MD+I$, one sees that for any integer $n \in \N$, $(MD)^n$ can be expanded in a sum of terms of the type $a_l M^lD^l$, $l \in \N$, $a_l \in \R$. The same holds for $Q(MD)$. Hence $A = \sum_{i, j}a_{ij}M^jD^i$ with $\left\{ j-i\, | \, a_{ij}\neq0 \right\} = \{0,k\}$.
  
   Let us prove the converse. We only treat the case  $\# \left\{ j-i \, | \, a_{ij}\neq0 \right\}=2$, the others being similar.
  Assume $A = \sum_{i, j}a_{ij}M^jD^i$ such that $\# \left\{ j-i  \, | \, a_{ij}\neq0 \right\} = \{k_1, k_2\}$ with $k_1>k_2$, is a Stein operator for $X$. Assume first that $k_2 \geq 0$. Then from the above, $M^j D^i \in M^{j-i} \mathcal T$, the latter set being either $M^{k_1} \mathcal T$ or $M^{k_2}\mathcal T$, depending on the value of $j-i$. But from \eqref{eq:commT}, $M^{k_1} \mathcal T = M^{k_1-k_2}\mathcal T M^{k_2}$ and $M^{k_2}\mathcal T =\mathcal T M^{k_2} $. Hence $A$ can be rewritten $L M^{k_2} - M^{k_1-k_2}K M^{k_2}$, with $L, K \in \mathcal T$. Simplifying by $M^{k_2}$ on the right (i.e., applying $A$ to $x \mapsto f(x) x^{-k_2}$ instead of $f$) yields the result.
  
  If $k_2<0$, one can multiply $A$ by $M^{-k_2}$ on the right and proceed in the same manner.
  \end{proof}

\begin{rmk}
  Assume $X$ admits a smooth density $p$, which solves the
  differential equation $B p = 0$ with
  $B = \sum_{i,j} b_{ij} M^j D^i$. Then, by duality (i.e.\ integration
  by parts; see Section \ref{sec:form-dens-char} for further detail),
  a Stein operator for $X$ is given by
  $A = \sum_{i,j} (-1)^i b_{ij} D^i M^j$. Then, in a similar manner as
  in the previous lemma, one can prove that $X$ satisfies Assumption~3
  if, and only if, $\# \left\{ j-i\, | \, b_{ij}\neq0 \right\}\le
  2$. In other words, the condition given in  Lemma~\ref{lma:building-blocks} for $X$ to
  satisfy Assumption~3 can be equivalently checked on the Stein
  operator $A$ or on the differential operator $B$ which cancels out
  the density of $X$.
\end{rmk}

One can specialize the result of Lemma \ref{lma:building-blocks} when the score of the distribution of $X$ is a rational fraction, which includes a wide class of classical distributions.  In this paper, $X$ will be a continuous random variable, and we use the terminology score function of $X$ to mean the logarithmic derivative of its probability density function.

\begin{cor}
Assume $X$ admits a score function of the form
\eq
\label{eq:score}
\rho(x):= \frac{p'(x)}{p(x)} = \frac{a \,x^k+b \,x^l}{c \,x^{k+1}+d \,x^{l+1}},
\qe
with $k, l \in \N$. Then $X$ satisfies Assumption~3.

Conversely,  if the score $\rho$ of $X$ is a rational fraction, and if
$X$ satisfies Assumption~3, then $\rho$ is of the form
\eqref{eq:score}. 
\end{cor}
\begin{proof}
Assume that
$$\frac{p'(x)}{p(x)}=\frac{\sum_{i=0}^n a_i x^i}{\sum_{j=0}^mb_j x^j}.$$
Starting from the canonical Stein operator $f \mapsto f'+ \frac{p'}{p} f$, and applying it to $(\sum_{j=0}^mb_j x^j)f(x)$, we have that $A = D\big(\sum_{j=0}^n b_j M^j\big) + \sum_{i=0}^na_i M^i$ is a Stein operator for $X$. Since $DM^j = j M^{j-1} + M^j D$, an application of Lemma \ref{lma:building-blocks} yields the result.
\end{proof}

Let us now focus on the class of Pearson distributions, that is
  the collection of all continuous probability distributions which
  satisfy (\ref{eq:score}) with $k=1, l=0$. We give here a
generalization of \cite{stein2}, Theorem 1, p.\ 65 proving that
Pearson distributions have polynomial Stein operators which can be
written in terms of operators in $\mathcal T$.

\begin{lma}[\cite{stein2},Theorem 1, p.\ 65]\label{sec:pears-distr-1}
  Let $X$ be  of Pearson type with score function 
\begin{equation}
\label{eq:31}
\rho(x):= \frac{p'(x)}{p(x) } =- \frac{ax-\ell}{\delta_2x^2+\delta_1x+\delta_0}
\end{equation} 
for some $a, \ell, \delta_0, \delta_1, \delta_2$ and $x$ on the
resulting support.  If $\delta_0=0$
then
\begin{equation}
  \label{eq:34}
  A =  M \delta_2T_{2-a/\delta_2} + \delta_1 T_{1+\ell/\delta_1}
\end{equation}
and if $\delta_0\neq0$  then 
  \begin{equation}\label{eq:32}
    A =  M^2 \delta_2T_{3-a/\delta_2}+\delta_1 M
    T_{1+\ell/\delta_1} +
   \delta_0 T_{1}
  \end{equation}
is a Stein operator for $X$. 
\end{lma}

\begin{proof}
 If $X$ is Pearson with log-derivative \eqref{eq:31} then denoting
 $-P_{\mathrm{num}}/P_{\mathrm{denom}}$ this ratio we see that  
$\left( f  P_{\mathrm{denom}}p  \right)'/p  = f
(P_{\mathrm{denom}}'-P_{\mathrm{num}}) + f'P_{\mathrm{denom}}$, i.e.
\begin{equation*}
 f'(x) \left( \delta_2x^2 + \delta_1 x + \delta_0\right) 
+ f(x) \left(  (2 \delta_2 + a) x + (\delta_1 - \ell) \right). 
\end{equation*}
This operator is integrable with respect to $p$ (with integral 0) for
all $f \in \mathcal{F}$.  Conclusion \eqref{eq:34} follows
immediately, while \eqref{eq:32} is obtained after replacing $f$ with
$xf$ and using $T_rM = M T_{r+1}$.
\end{proof}

\begin{exm} \label{ex:building-blocks}
The normal distribution $N(\mu, \sigma^2)$ falls into this class.  It has  log-derivative
  $\rho(x) = -(x-\mu)/\sigma^2$ on $\R$: $a=1$, $\ell=\mu$,
  $\delta_1 = \delta_2 = 0$ and $\delta_0 = \sigma^2$ and
  \eqref{eq:32} applies leading to $A = \sigma^2T_1 + \mu M - M^2$
  (recall that $r T_{1/r} \to I$ as $r\to 0$).  Assumption~3 is satisfied if and only if $\mu=0$; in a future work \cite{gms18} we shall introduce a technique for obtaining Stein operators for products of a class of distributions, which includes the non centered normal distribution, that do not satisfy Assumption~3.   Other examples that are in the class of Lemma \ref{sec:pears-distr-1} include the gamma, beta, Student's $t$, and inverse-gamma distributions, all of which satisfy Assumption~3.  The resulting Stein operators (which are not new to the literature) are given in Appendix \ref{appendix:list}.
  \end{exm}

\subsection{The algebra for products of distributions}
\label{sec:prod-distr}

We first note how Proposition \ref{prop:1} is easily generalised to
the product of $n$ independent random variables, by induction. More precisely, if
$(X_i)_{1\leq i \leq n}$ are independent random variables with
respective Stein operator $L_i - M^p K_i$, if all the operators
$\{L_i, K_i\}_{1\leq i \leq n}$ commute with each other and with the
$\tau_a$, $a\in \R$, then a Stein operator for $\prod_{i=1}^n X_i$ is
$$\prod_{i=1}^{n}L_i - M^p \prod_{i=1}^n K_i.$$ 

The main drawback of Proposition \ref{prop:1} is that we assume the
same power of $M^p$ appears in both operators. As such, the
Proposition cannot be applied for instance for the product of a gamma
(for which $p=1$, see Appendix \ref{appendix:list}) and a centered
normal (for which $p=2$, see Appendix \ref{appendix:list}). In the
following Lemma and Proposition, we show how to bypass this
difficulty: one can build another Stein operator for $X$ with the
power $p$ multiplied by an arbitrary integer $k$ (even though by doing
so, one increases the order of the operator). Here we restrict
ourselves to the case where the $L_i$ and $K_i$ operators are products
of operators $T_\alpha$ and we make use of the relation \eqref{eq:recT}.

\begin{lma}
\label{lma:1}
Assume $X$ has a Stein operator of the form
\eq
\label{eq:axlma}
A_X = a\prod_{i=1}^n T_{\alpha_i}- b M^p \prod_{i=1}^m
T_{\beta_i}.  \qe Then, for every $k \geq 1$, 
$$a^k \prod_{i=1}^n \prod_{j=0}^{k-1} T_{\alpha_i+jp} - b^k M^{kp}
\prod_{i=1}^m \prod_{j=0}^{k-1} T_{\beta_i+jp}$$
is a Stein operator for $X$.
\end{lma}
\begin{proof}
  We prove the result by induction on $ k$. By assumption, it is true
  for $k=1$. Then, using the recurrence hypothesis and
  \eqref{eq:recT},
\begin{align*}
\E\bigg[ a^{k+1} \prod_{j=0}^k\prod_{i=1}^n T_{\alpha_i + j k} f(X) \bigg]& = \E\bigg[ a a^{k}  \prod_{j=0}^{k-1}\prod_{i=1}^n T_{\alpha_i + j p}\bigg( \prod_{i=1}^n T_{\alpha_i+kp} f\bigg)(X) \bigg]\\
& = \E\bigg[ a b^{k} M^{kp}\prod_{j=0}^{k-1}\prod_{i=1}^m T_{\beta_i + j p}\bigg( \prod_{i=1}^n T_{\alpha_i+kp} f\bigg)(X)\bigg]\\
& = \E\bigg[ a b^{k} \prod_{i=1}^n T_{\alpha_i}M^{kp}\prod_{j=0}^{k-1}\prod_{i=1}^m T_{\beta_i + j p} f(X)\bigg]
\end{align*}
\begin{align*}
& = \E\bigg[ b^{k+1} M^p\prod_{i=1}^mT_{\beta_i}M^{kp}\prod_{j=0}^{k-1}\prod_{i=1}^m T_{\beta_i + j p} f(X)\bigg]\\
& = \E\bigg[ b^{k+1} M^{kp}\prod_{j=0}^{k}\prod_{i=1}^m T_{\beta_i + j p} f(X)\bigg],
\end{align*}
which proves our claim.
\end{proof}

Lemma \ref{lma:1} leads to the following rule of thumb for the
problem of finding a Stein operator for a product of independent
random variables $X$ and $Y$ with Stein operators
$A_X = a\prod_{i=1}^n T_{\alpha_i}- b M^p \prod_{i=1}^m
T_{\beta_i}$ and
$A_Y = a'\prod_{i=1}^{n'} T_{\alpha'_i}- b' M^{p'}
\prod_{i=1}^{m'} T_{\beta'_i}$, with $p \neq p'$. Apply Lemma
\ref{lma:1} to $X$ with $k = p'$ and to $Y$ with $k = p$ to get Stein
operators for $X$ and $Y$ of the form of Proposition \ref{prop:1}, but
with $p$ replaced by $pp'$. Then apply the Proposition. As an illustration, one can prove the following.

\begin{prop}
\label{prop:2}
Assume $X,Y$ are random variables with respective Stein operators
\begin{eqnarray*}
A_X &=& a_1 T_{\alpha_1} - a_2 M^p T_{\alpha_2},\\
A_Y &=& b_1 T_{\beta_1} - b_2 M^q T_{\beta_2}, 
\end{eqnarray*}
where $p,q\in \N$ and
$\alpha_1,\alpha_2,\beta_1,\beta_2 \in \R\cup\{\infty\}$ (and Remark
\ref{rmk:infiniteparam} applies in case  any of the $\alpha_i$ or
$\beta_i$, $i=1, 2$ is set
to $+\infty$). Let $m$ be
the least common multiple of $p$ and $q$ and write $m = k_1p = k_2
q$.
Then, if $X$ and $Y$ are independent, 
\begin{equation*} a_1^{k_1} b_1^{k_2}
\prod_{i=0}^{k_1-1}T_{\alpha_1+ip}\prod_{i=0}^{k_2-1}T_{\beta_1+i q} -
M^m a_2^{k_1}b_2^{k_2}
\prod_{i=0}^{k_1-1}T_{\alpha_2+ip}\prod_{i=0}^{k_2-1}T_{\beta_2+i q}
\end{equation*}
is a Stein operator for $XY$. 
\end{prop}
\begin{proof}
Apply Lemma \ref{lma:1} with $k_1$ and $k_2$ to get, for all $f \in \mF$,
$$\E\bigg[ a_1^{k_1} \prod_{j=0}^{k_1-1}T_{\alpha_1+jp} f(X) \bigg] = \E\bigg[ a_2^{k_1} M^{m} \prod_{j=0}^{k_1-1}T_{\alpha_2+jp}f(X) \bigg],$$
and
$$\E\bigg[ b_1^{k_2} \prod_{j=0}^{k_2-1}T_{\beta_1+jp} f(Y) \bigg] = \E\bigg[ b_2^{k_2} M^{m} \prod_{j=0}^{k_2-1}T_{\beta_2+jp}f(Y) \bigg].$$
Then the proof follows from an application of Proposition \ref{prop:1}.
% with
%\begin{align*}
%L_X &= a_1^{k_1} T_{\alpha_1}\ldots T_{\alpha_1+(k_1-1) p}\\
%K_X &= a_2^{k_1} T_{\alpha_2}\ldots T_{\alpha_2+(k_1-1)p}\\
%L_Y &= b_1^{k_2} T_{\beta_1}\ldots T_{\beta_1+(k_2-1) q}\\
%K_Y &= b_2^{k_2} T_{\beta_2}\ldots T_{\beta_2+(k_2-1)q}.
%\end{align*}
\end{proof}

%We note that we can applying Proposition \ref{prop:2} when some of the
%$T_{\alpha}$'s and $T_{\beta}$'s are replaced by the identity operator
%$I$.  This can be seen because $cT_{1/c}=cMD+I$.  Taking the limit
%gives that $I=\lim_{c\rightarrow0}T_{1/c}$.  This observation yields,
%as an example, the following corollary.
%
%\begin{cor}Assume $X,Y$ are r.v. with respective Stein operators
%\begin{eqnarray}
%A_X =& a_1 T_{\alpha} - a_2 M^p \label{eq:4}\\
%A_Y =& b_1 T_{\beta} - b_2 M^q, \label{eq:3}
%\end{eqnarray}
%where $p,q\in \N$ and $\alpha,\beta \in \R$. Let $m$ be the least
%common multiple of $p$ and $q$ and write $m = k_1p = k_2 q$. If $X$
%and $Y$ are independent, then \eq a_1^{k_1} b_1^{k_2}
%\prod_{i=0}^{k_1-1}T_{\alpha+ip}\prod_{i=0}^{k_2-1}T_{\beta+i q} -
%a_2^{k_1}b_2^{k_2} M^m.  \qe is a (weak) Stein operator for $XY$.
%\end{cor}

%We note that we can applying Proposition \ref{prop:2} when some of the $T_{\alpha}$'s and $T_{\beta}$'s are replaced by the identity operator $I$.  This can be seen because $cT_{1/c}=cMD+I$.  Taking the limit gives that $I=\lim_{c\rightarrow0}T_{1/c}$.  Applying Proposition \ref{prop:2} in this limiting form this allows us to treat cases in which some of the $T_{\alpha}$'s and $T_{\beta}$'s are replaced by the identity operator $I$.

\begin{rmk}
Let us give an example when one of the $T$ operators is the identity. We have that if $X,Y$ are random variables with respective Stein operators
\begin{eqnarray*}
A_X &=& a_1 T_{\alpha} - a_2 M^p, \\
A_Y &=& b_1 T_{\beta} - b_2 M^q, 
\end{eqnarray*}
then, with the same notation, a Stein operator for $XY$ is
 \[ a_1^{k_1} b_1^{k_2}
\prod_{i=0}^{k_1-1}T_{\alpha+ip}\prod_{i=0}^{k_2-1}T_{\beta+i q} -
a_2^{k_1}b_2^{k_2} M^m.  \]
\end{rmk}

\subsection{The algebra for powers and inverse distributions}
\label{sec:powers-inverse-distr}

In this section, we consider (not necessarily integer nor positive)
powers of $X$'s and this implies we shall need to modify Assumption 2
from Section \ref{sec:classfunctions} to ensure that all operators are
well defined. If $X$ takes values a.s.\ in $\R \backslash \{0\}$, then
we restrict to test functions
$f \in C_0^{\infty}(\R\backslash \{0\})$. If $X$ takes values a.s.\ in
$(0,\infty)$, then we further restrict to test functions
$f \in C_0^{\infty}((0, \infty))$.  Likewise, we extend some notations
from Section \ref{sec:notations}. We extend the definition of $M^a$ to
$a \in \mathbb Z$ by $M^a f (x) = x^a f(x)$, $x \neq 0$. Further we
extend this definition to $a\in \R$ for $x \in (0, \infty)$.

Let us first note a result concerning powers.  
Let $P_a$ be defined by
$P_af(x)=f(x^a)$.  For $a\not=0$, we have that $T_rP_a=aP_aT_{r/a}$, since
\begin{align*}T_rP_af(x) &=x\cdot ax^{a-1}f'(x^a)+rf(x^a)=ax^af'(x^a)+rf(x^a)\\
& =a(x^af'(x^a)+(r/a)f(x^a))=aP_aT_{r/a}f(x).
\end{align*}
This result allows us to easily obtain Stein operators for
powers of random variables and inverse distributions.  Suppose $X$ has
Stein operator
\begin{equation*}A_X=aT_{\alpha_1}\cdots T_{\alpha_n}-bM^qT_{\beta_1}\cdots T_{\beta_m}.
\end{equation*} 
We can write down a Stein operator for $X^\gamma$ immediately (if $X$ takes negative values, we restrict to positive or integer-valued $\gamma$):
\begin{align}A_{X^\gamma}&=aT_{\alpha_1}\cdots T_{\alpha_n}P_\gamma-bM^qT_{\beta_1}\cdots T_{\beta_m}P_\gamma \nonumber \\
&=a\gamma^nP_\gamma T_{\alpha_1/\gamma}\cdots T_{\alpha_n/\gamma}-b\gamma^mM^qP_\gamma T_{\beta_1/\gamma}\cdots T_{\beta_m/\gamma} \nonumber \\
\label{powereqn1}&=a\gamma^nP_\gamma T_{\alpha_1/\gamma}\cdots T_{\alpha_n/\gamma}-b\gamma^mP_\gamma M^{q/\gamma} T_{\beta_1/\gamma}\cdots T_{\beta_m/\gamma}.
\end{align}
Applying $P_{1/\gamma}$ on the left of (\ref{powereqn1}) gives the following Stein operator for the random variable $X^\gamma$:
\begin{equation}\label{powereqn2}\tilde{A}_{X^\gamma}=a\gamma^{n} T_{\alpha_1/\gamma}\cdots T_{\alpha_n/\gamma}-b\gamma^m M^{q/\gamma} T_{\beta_1/\gamma}\cdots T_{\beta_m/\gamma},
\end{equation}
as $P_{1/\gamma}P_{\gamma}=I$.

From (\ref{powereqn2}) we immediately obtain, for example, the classical $\chi_{(1)}^2$ Stein operator $T_{1/2}-\frac{1}{2}M$ from the standard normal Stein operator $T_1-M^2$.  However, in certain situations, a more convenient form of the Stein operator may be desired.  To illustrate this, we consider the important special case of inverse distributions.  Here $\gamma=-1$, which yields the following Stein operator for $1/X$:
\begin{equation*}a(-1)^n T_{-\alpha_1}\cdots T_{-\alpha_n}-b(-1)^m M^{-q} T_{-\beta_1}\cdots T_{-\beta_m}.
\end{equation*}
To remove the singularity, we multiply on the right by $M^{-1}$ to get
\begin{align}A_{1/X}&=a(-1)^n T_{-\alpha_1}\cdots T_{-\alpha_n}M^q-b(-1)^m M^{-q} T_{-\beta_1}\cdots T_{-\beta_m}M^q \nonumber \\
&=a(-1)^n M^qT_{q-\alpha_1}\cdots T_{q-\alpha_n}-b(-1)^m  T_{q-\beta_1}\cdots T_{q-\beta_m}\nonumber.
\end{align}
Cancelling constants gives the Stein operator
\begin{align}\label{inveqn1}\tilde{A}_{1/X}=b T_{q-\beta_1}\cdots T_{q-\beta_m}-(-1)^{m+n}a M^qT_{q-\alpha_1}\cdots T_{q-\alpha_n} .
\end{align}

\section{Applying the algebra to find new Stein operators}
\label{sec:applications-1}
Starting from the classical Stein operators of the centered normal, gamma, beta, Student's $t$, inverse-gamma, PRR, variance-gamma (with $\theta=0$ and $\mu=0$), and generalized gamma distributions, we use the results of Section \ref{sec:prod-distr} to derive new operators for the (possibly mixed) products of these distributions. The operators of the aforementioned distributions are summed up in Appendix \ref{appendix:list}. Stein operators for any mixed product of independent copies of such random variables are attainable through a direct application of Proposition \ref{prop:2}.  We give some examples below. 
%We derive new Stein operators for product and quotient distributions and show that a number of Stein operators from the existing literature can be derived easily using the algebra of Stein operators developed in Sections \ref{sec:prod-distr} and \ref{sec:powers-inverse-distr}.  It is indeed interesting to note just how many of the Stein operators from the existing literature can be derived from a simple application of our algebra of Stein operators to just a small number of standard and easy to obtain Stein operators. 

%Stein operators for the considered distributions are available in the literature (references are given in Appendix \ref{appendix:list}). By using similar techniques as the one explained in the normal case, one can build new ones that are of the form needed to apply our Propositions. , the operators we obtain. The distribution of each considered random variable is one of the distributions of Table \ref{tab:2} of Appendix \ref{appendix:list}.

\subsection{Mixed products of centered normal and gamma random variables}

\label{subsec:1}

Stein operators for (mixed) products of independent central normal, beta and gamma random variables were obtained by \cite{gaunt pn, gaunt ngb}.  Here we demonstrate how these Stein operators can be easily derived by an application of our theory (we omit the beta distribution for reasons of brevity).  Let $(X_i)_{1\leq i \leq n}$ and $(Y_j)_{1\leq j \leq m}$ be independent random variables and assume $X_i \sim \mathcal N(0,\sigma^2)$ and $Y_j \sim \Gamma(r_j,\lambda_j)$.  The random variables $X_i$ and $Y_j$ admit the following Stein operators:
\begin{eqnarray}\label{sec61}A_{X_i}&=&\sigma_i^2T_1-M^2,\\
\label{sec62}A_{Y_j}&=&T_{r_j}-\lambda_jM.
\end{eqnarray}
A repeated application of Proposition \ref{prop:2} now gives the following Stein operators:
\begin{eqnarray}\label{ngp1}A_{X_1\cdots X_n}&=&\sigma_1^2\cdots\sigma_n^2T_1^n-M^2, \\
\label{ngp2}A_{Y_1\cdots Y_m}&=&T_{r_1}\cdots T_{r_m}-\lambda_1\cdots \lambda_m M, \\
\label{ngp3}A_{X_1\cdots X_nY_1\cdots Y_m}&=&\sigma_1^2\cdots\sigma_n^2T_1^nT_{r_1}\cdots T_{r_m}T_{r_1+1}\cdots T_{r_m+1}-\lambda_1\cdots \lambda_m M^2.
\end{eqnarray}
The product gamma Stein operator (\ref{ngp2}) is in exact agreement with the one obtained by \cite{gaunt ngb}.  However, the Stein operators (\ref{ngp1}) and (\ref{ngp3}) differ slightly from those of \cite{gaunt pn, gaunt ngb}, because they act on different functions.  Indeed, the product normal Stein operator given in \cite{gaunt pn} is $\tilde{A}_{X_1\cdots X_n}=\sigma_1^2\cdots\sigma_n^2DT_0^n-M$, but multiplying through on the right by $M$ yields (\ref{ngp1}).  The same is true of the mixed product operator (\ref{ngp3}), which is equivalent to the mixed normal-gamma Stein operator of \cite{gaunt ngb} multiplied on the right by $M$.  We refer to Appendix \ref{appendix:list} where this idea is expounded.

Finally, we note that whilst the operators (\ref{ngp1}) and (\ref{ngp2}) are of orders $n$ and $m$, respectively, the mixed product operator (\ref{ngp3}) is of order $n+2m$, rather than order $n+m$ which one may at first expect.  This a consequence of the fact that the powers of $M$ in the Stein operator (\ref{sec61}) and (\ref{sec62}) differ by a factor of 2.

\subsection{Mixed product of Student and variance-gamma random variables}

\label{sec:product-1}

Let $(X_i)_{1\leq i \leq n}$ and $(Y_j)_{1\leq j \leq m}$ be independent random variables and assume $X_i \sim t_{\nu_i}$ follows Student's $t$-distribution with $\nu$ degrees of freedom and $Y_j \sim \mathrm{VG}(r_j,0,\sigma_j,0)$; the p.d.f$.$s of these distributions are given in Appendix \ref{appendix:list}. $X_i$ and $Y_j$ admit Stein operators of the form:
\begin{eqnarray}
A_{X_i} &=& \nu_i T_1 + M^2 T_{2-\nu_i},\nonumber\\
\label{vgsteingms}A_{Y_j} &=& \sigma_j^2T_1T_{r_j}-M^2.
\end{eqnarray}
%The operator $A_{Y_j}$ is obtained by multiplying on the right by $M$ the usual $\mathrm{VG}(r_j,0,\sigma_j,0)$ Stein operator $\tilde{A}_{Y_j} = \sigma_j^2T_{r_j}D-M$.  
Note that one cannot apply Proposition \ref{prop:1} to the
$\mathrm{VG}(r,\theta,\sigma,0)$ Stein operator
$\sigma^2T_1T_{r}+2\theta MT_{r/2}-M^2$, because Assumption~3 is not satisfied. 

Applying recursively Proposition \ref{prop:1}, we obtain the following Stein operators:
\begin{eqnarray}A_{X_1\cdots X_n}&=&\nu_1 \ldots \nu_n  T_1^{n} - (-1)^{n} M^2 T_{2-\nu_1} \ldots T_{2-\nu_n},\nonumber \\
\label{ngp21}A_{Y_1\cdots Y_m}&=& \sigma_1^2 \ldots \sigma_m^2 T_1^{m} T_{r_1} \ldots T_{r_m} -  M^2, \\
A_{X_1\cdots X_nY_1\cdots Y_m}&=& \nu_1 \ldots \nu_n \sigma_1^2 \ldots \sigma_m^2 T_1^{n+m} T_{r_1} \ldots T_{r_m} - (-1)^{n} M^2 T_{2-\nu_1} \ldots T_{2-\nu_n}.\nonumber
\end{eqnarray}

As an aside, note that (\ref{vgsteingms}) can be obtained by applying Proposition \ref{prop:1} to the Stein operators
\begin{equation*}A_X=\sigma^2T_1-M^2, \quad A_Y=T_r-M^2,
\end{equation*}
where $X$ and $Y$ are independent.  We can identify $A_X$ as the Stein operator for a $\mathcal{N}(0,\sigma^2)$ random variable and $A_Y$ as the Stein operator of the random variable $Y=\sqrt{V}$ where $V\sim \Gamma(r/2,1/2)$.  Since the variance-gamma Stein operator is characterizing (see \cite{gaunt vg}, Lemma 3.1), it follows that that $Z\sim \mathrm{VG}(r,0,\sigma,0)$ is equal in distribution to $X\sqrt{V}$.  This representation of the $ \mathrm{VG}(r,0,\sigma,0)$ distribution can be found in \cite{barn}.  This example demonstrates that by characterizing probability distributions, Stein operators can be used to derive useful properties of probability distributions; for a further discussion on this general matter see Section \ref{sec:form-dens-char}.

The Stein operator (\ref{ngp21}) will be used in Section \ref{sec:form-dens-char} in a derivation of the formula for the p.d.f$.$ of the product of independent $\mathrm{VG}(r,0,\sigma,0)$ random variables.  As an example, following some straightforward calculations, we write down explicitly the Stein operator for the case $m=2$ and $\sigma_1=\sigma_2=1$:
%We note that
%\begin{align*}T_rT_sf(x)&=x^2f''(x)+(1+r+s)xf'(x)+rsf(x), \\
%T_rT_sT_tf(x)&=x^3f^{(3)}(x)+(3+r+s+t)x^2f''(x)+(1+r+s+t+1+r+s+t)xf'(x)\\
%&\quad+rstf(x), \\
%T_rT_sT_tT_uf(x)&=xf^{(4)}(x)+(6+r+s+t+u)x^3f^{(3)}(x)+(10+3(r+s+t+u)\\
%&\quad+rs+rt+ru+st+su+tu)x^2f''(x)+(1+r+s+t+u+rs+rt+ru\\
%&\quad+st+su+tu+rst+stu+rtu+rsu)xf'(x)+rstuf(x),
%\end{align*}
%and therefore
\begin{align*}A_{Y_1Y_2}f(x)&=x^4f^{(4)}(x)+(8+r_1+r_2)x^3f^{(3)}(x)+(17+5r_1+5r_2+r_1r_2)x^2f''(x)\\
&\quad+(4+4r_1+4r_2+3r_1r_2)xf'(x)+(r_1r_2-x^2)f(x).
\end{align*}

\subsection{PRR distribution}
\label{sec:prr-distribution}
A Stein operator for the PRR distribution is given by
\begin{equation}\label{prrop}A_{s}f(x)=sT_1T_2f(x)-M^2T_{2s}f(x)=sx^2f''(x)+(4sx-x^3)f'(x)+2s(1-x^2)f(x),
\end{equation}
see Appendix \ref{appendix:list}. We now exhibit a neat derivation of this Stein operator by an application of Sections \ref{sec:prod-distr} and \ref{sec:powers-inverse-distr}.  Let $X$ and $Y$ be independent random variables with distributions
\begin{equation*}X\sim  \begin{cases} \mathrm{Beta}(1,s-1), & \quad \text{if $s>1$,} \\
\mathrm{Beta}(1/2,s-1/2), & \quad  \text{if $1/2<s\leq 1$,} \end{cases}
\end{equation*}
and
\begin{equation*}Y\sim  \begin{cases} \Gamma(1/2,1), & \quad \text{if $s>1$,} \\
\mathrm{Exp}(1), & \quad  \text{if $1/2<s\leq 1$.} \end{cases}
\end{equation*}
Then it is known that $\sqrt{2sXY}\sim K_s$ (see \cite{pekoz}, Proposition 2.3).

If $s>1$, then we have the following Stein operators for $X$ and $Y$:
\begin{equation*}A_X=T_1-MT_s, \quad A_Y=T_{1/2}-M,
\end{equation*}
and, for $1/2<s\leq1$,  we have the following Stein operators for $X$ and $Y$:
\begin{equation*}A_X=T_{1/2}-MT_s, \quad A_Y=T_{1}-M.
\end{equation*}
Using Proposition \ref{prop:2}, we have that, for all $s>1/2$,
\begin{equation*}A_{XY}=T_{1/2}T_1-MT_s.
\end{equation*}
From (\ref{powereqn2}) we obtain the Stein operator
\begin{equation*}A_{\sqrt{XY}}=T_1T_2-2M^2T_{2s},
\end{equation*}
which on rescaling by a factor of $\sqrt{2s}$ yields the operator (\ref{prrop}).

%\subsubsection{Product of variance-gamma $(r,0,\sigma,0)$ random variables}

%\label{sec:prod-vari-gamma}
%Again the proposition applies. A Stein operator for this type of r.v. is
%$$\sigma^2 M D^2 + \sigma^2 r D - M,$$
%multiplying by $M$ on the right, another form is
%\begin{align}
%&\sigma^2 M D^2 M + \sigma^2 r D M - M^2\nonumber \\
%=&\sigma^2 MD(MD+I) + \sigma^2 r (MD+I) - M^2\nonumber\\
%=& \sigma^2 (MD+I)(MD+rI)-M^2\nonumber\\
%\label{prodvgstein}=&\sigma^2 T_1 T_r - M^2
%\end{align}
%so that \eqref{eq:1} holds with $L_1 = T_1 T_r$, $L_2 = I$ and $p=2$. Thus we have the following Stein operator for a product of $n$ Variance-Gammas with parameters $(r_i,0,\sigma_i,0)$:
%$$\sigma_1^2\ldots\sigma_n^2 T_1^n T_{r_1}\ldots T_{r_n} - M^2. $$

\subsection{Inverse and quotient distributions}

From (\ref{inveqn1}) we can write down inverse distributions for many standard distributions.  First, suppose $X\sim \mathrm{Beta}(a,b)$.  Then a Stein operator for $1/X$ is
\begin{equation}\label{betainv}A_{1/X}=T_{1-a-b}-MT_{1-a}.
\end{equation}
Now, let $X_1\sim\mathrm{Beta}(a_1,b_1)$ and $X_2\sim\mathrm{Beta}(a_2,b_2)$ be independent.  Then using Proposition \ref{prop:2} applied to the Stein operator (\ref{betainv}) for $1/X$ and the beta Stein operator, we have the following Stein operator for $Z=X_1/X_2$:
\begin{equation}\label{betainvstein}A_Z=T_{a_1}T_{1-a_2-b_2}-MT_{a_1+b_1}T_{1-a_2},
\end{equation}
which is a second order differential operator.

Let us consider the inverse-gamma distribution.  Let $X\sim\Gamma(r,\lambda)$, then the gamma Stein equation is $A_X=T_r-\lambda M$. From (\ref{inveqn1}) we can obtain a Stein operator for $1/X$ (an inverse-gamma random variable):
\begin{equation*}A_{1/X}=MT_{1-r}-\lambda I.
\end{equation*}
If $X_1\sim\Gamma(r_1,\lambda_1)$ and $X\sim\Gamma(r_2,\lambda_2)$, we have from the above operator and Proposition \ref{prop:2}, the following Stein operator for $Z=X_1/X_2$:
\begin{equation}\label{invgamma}A_Z=\lambda_1MT_{1-r_2}+\lambda_2T_{r_1},
\end{equation}
which is a first order differential operator.  As a special case, we can obtain a Stein operator for the $F$-distribution with parameters $d_1>0$ and $d_2>0$.  This is because $Z\sim F(d_1,d_2)$ is equal in distribution to $\frac{X_1/d_1}{X_2/d_2}$, where $X_1\sim \chi_{(d_1)}^2$ and $X_2\sim\chi_{(d_2)}^2$ are independent.  Now applying (\ref{invgamma}) and rescaling to take into account the factor $d_1/d_2$ gives the following Stein operator for  $Z$:
\begin{equation}\label{fdist}A_Z=d_1MT_{1-d_2/2}+d_2T_{d_1/2}.
\end{equation}
As this Stein operator seems to be new to the literature and may prove useful in applications due to the importance of the $F$-distribution in statistics, we write out the operator explicitly:
\begin{equation*}A_Zf(x)=(d_1x^2+d_2x)f'(x)+(d_1d_2/2+d_1(1-d_2/2)x)f(x).
\end{equation*} 

%In is interesting to note that (\ref{fdist}) is a first order operator,
%which is lower than the second order operator (\ref{betainvstein}).
%Due to the duality between Stein operators and differential equations
%for densities, we would expect to obtain a first order operator
%because the density of the $F$-distribution,
%\begin{equation*}p(x)=\frac{1}{B(\frac{d_1}{2},\frac{d_2}{2})}\bigg(\frac{d_1}{d_2}\bigg)^{d_1/2}x^{d_1/2-1}\bigg(1+\frac{d_1}{d_2}x\bigg)^{-(d_1+2_2)/2}, \quad x>0,
%\end{equation*}
%is an elementary function, which satisfies a first order linear differential equation.  Here, $B(\frac{d_1}{2},\frac{d_2}{2})$ is the beta function.

One can also easily derive the generalized gamma Stein operator from the gamma Stein operator.  The Stein operator for the $\mathrm{GG}(r,\lambda,q)$ distribution is given by $T_r-q\lambda^qM^q$.  Using the relationship $X\stackrel{\mathcal{L}}{=}(\lambda^{1-q}Y)^{1/q}$ for $X\sim\mathrm{GG}(r,\lambda,q)$ and $Y\sim \Gamma(r/q,\lambda)$ (see \cite{pekoz3}) together with (\ref{powereqn2}) and a rescaling, we readily recover the generalized gamma Stein operator from the usual gamma Stein operator.

As a final example, we note that we can use Proposition \ref{prop:2} to obtain a Stein operator for the ratio of two independent standard normal random variables.  A Stein operator for the standard normal random variable $X_1$ is $T_1-M^2$ and we can apply (\ref{inveqn1}) to obtain the following Stein operator for the random variable $1/X_1$:
\begin{equation*}A_{1/X_1}=M^2T_1-I.
\end{equation*}
Hence a Stein operator for the ratio of two independent standard normals is 
\begin{equation*}A=(1+M^2)T_1,
\end{equation*} 
which is the Stein operator for the Cauchy distribution (a special case of the Student's $t$ Stein operator of \cite{schoutens}), as one would expect.

\section{Duals of Stein operators and densities of product distributions}
\label{sec:form-dens-char}

Fundamental methods, based on the Mellin integral transform, for
deriving formulas for densities of product distributions were
developed by \cite{springer66, springer}.  In \cite{springer},
formulas, involving the Meijer $G$-function, were obtained for
products of independent centered normals, and for mixed products of
beta and gamma random variables.  However, for other product
distributions, applying the Mellin inversion formula can lead to
intractable calculations.

In this section, we present a novel method for deriving formulas for
densities of product distributions based on the duality between Stein
operators and ODEs satisfied by densities.  Our approach builds on
that of \cite{gaunt ngb} in which a duality argument was used to
derive a new formula for the p.d.f$.$ of a mixed product of mutually
independently centered normal, beta and gamma random variables
(deriving such a formula using the Mellin inversion formula would have
required some very involved calculations).  We apply this method to
derive a new formula for the p.d.f. of the product of $n$ independent
$\mathrm{VG}(r,0,\sigma,0)$ random variables. We begin with a duality
lemma whose proof is a generalisation of the argument given in Section
3.2 of \cite{gaunt ngb}.

 \begin{lma}\label{dulem}Let $Z$ be a random variable with density $p$
   supported on an interval $[a, b]\subseteq\R$.  Let  
 \begin{equation}\label{sec5op}Af(x)=T_{r_1}\cdots T_{r_n}f(x)-bx^qT_{a_1}\cdots T_{a_m}f(x),
\end{equation}
 and suppose that
 \begin{equation}\label{exp11}\E[Af(Z)]=0
 \end{equation}
 for all $f\in C^k([a,b])$, where $k=\max\{m,n\}$, such that
 \begin{enumerate}
 \item $x^{q+1+i+j}p^{(i)}(x)f^{(j)}(x)\rightarrow 0$, as
   $x\rightarrow a$ and as $x\rightarrow b$, for all $i,j$ such that
   $0\leq i+j\leq m$;
 \item $x^{1+i+j}p^{(i)}(x)f^{(j)}(x)\rightarrow 0$, as
   $x\rightarrow a$ and as $x\rightarrow b$, for all $i,j$ such that
   $0\leq i+j\leq n$.
 \end{enumerate}
 (We denote this class of functions by $\mathcal{C}_p$).  Then $p$ satisfies the differential equation
 \begin{equation}\label{lolcofe11}T_{1-r_1}\cdots T_{1-r_n}p(x)-b(-1)^{m+n}x^q T_{q+1-a_1}\cdots T_{q+1-a_m}p(x)=0.
 \end{equation}
\end{lma}

\begin{rmk}The class of functions $\mathcal{C}_p$ consists of all
  $f\in C^k([a,b])$, where $k=\max\{m,n\}$, that satisfy particular
  boundary conditions at $a$ and $b$.  Note that when
  $(a,b)=\R$ the class includes the set of all functions on
  $\R$ with compact support that are $k$ times differentiable.
  The class $\mathcal{C}_p$ suffices for the purpose of deriving the
  differential equation (\ref{lolcofe11}), although we expect that for
  particular densities (such as the beta distribution) the conditions
  on $f$ could be weakened.
 \end{rmk}

 \begin{proof}
 We begin by writing the expectation (\ref{exp11}) as
 \begin{equation}\label{integral1}\int_{a}^b \big\{T_{r_1}\cdots T_{r_n}f(x)-bx^qT_{a_1}\cdots T_{a_m}f(x)\big\}p(x)\,\mathrm{d}x=0,
 \end{equation}
 which exists if $f\in \mathcal{C}_p$.  In arriving at the differential equation (\ref{lolcofe11}), we shall apply integration by parts repeatedly.  To this end, it is useful to note the following integration by parts formula.  Let $\gamma\in\R$ and suppose that $\phi$ and $\psi$ are differentiable.  Then
 \begin{align}\int_a^b x^\gamma\phi(x)T_r\psi(x)\,\mathrm{d}x&=\int_{a}^{b}x^\gamma\phi(x)\{x\psi'(x)+r\psi(x)\}\,\mathrm{d}x =\int_{a}^{b}x^{\gamma+1-r}\phi(x)\frac{\mathrm{d}}{\mathrm{d}x}(x^r\psi(x))\,\mathrm{d}x\nonumber \\
 &=\Big[x^{\gamma+1}\phi(x)\psi(x)\Big]_{a}^{b}-\int_{a}^{b}x^r\psi(x)\frac{\mathrm{d}}{\mathrm{d}x}(x^{\gamma+1-r}\phi(x))\,\mathrm{d}x\nonumber \\
 \label{skts}&=\Big[x^{\gamma+1}\phi(x)\psi(x)\Big]_{a}^{b}-\int_{a}^{b}x^\gamma\psi(x)T_{\gamma+1-r}\phi(x)\,\mathrm{d}x,
 \end{align}
 provided the integrals exist.  

 We now return to equation (\ref{integral1}) and use the integration by parts and formula (\ref{skts})  to obtain a differential equation that is satisfied by $p$.  Using (\ref{skts}) we obtain
 \begin{align*}\int_{a}^{b}x^qp(x)T_{a_1}\cdots T_{a_m}f(x)\,\mathrm{d}x
 &=\Big[x^{q+1}p(x)T_{a_2}\cdots T_{a_m}f(x)\Big]_{a}^{b}\\
 &\quad-\int_{a}^{b}x^qT_{q+1-a_1}p(x)T_{a_2}\cdots T_{a_m}f(x)\,\mathrm{d}x \\
 &=-\int_{a}^{b}x^qT_{q+1-a_1}p(x)T_{a_2}\cdots T_{a_m}f(x)\,\mathrm{d}x,
 \end{align*}
 where we used condition (i) to obtain the last equality.  By a repeated application of integration by parts, using formula (\ref{skts}) and condition (i), we arrive at
 \begin{align*}\int_{a}^{b}x^qp(x)T_{a_1}\cdots T_{a_m}f(x)\,\mathrm{d}x=(-1)^m\int_{a}^{b}x^qf(x)T_{q+1-a_1}\cdots T_{q+1-a_m}p(x)\,\mathrm{d}x.
 \end{align*}
 By a similar argument, this time using formula (\ref{skts}) and condition (ii),  we obtain
 \begin{align*}\int_{a}^{b}p(x)T_{r_1}\cdots T_{r_n}f(x)\,\mathrm{d}x=(-1)^n\int_{a}^{b}f(x)T_{1-r_1}\cdots T_{1-r_n}p(x)\,\mathrm{d}x.
 \end{align*}
 Putting this together we have that
 \begin{align}\label{integral2}\int_{a}^b \{(-1)^nT_{1-r_1}\cdots T_{1-r_n}p(x)-b(-1)^mx^q T_{q+1-a_1}\cdots T_{q+1-a_m}p(x)\}f(x)\,\mathrm{d}x=0
 \end{align}
 for all $f\in\mathcal{C}_p$.  Since (\ref{integral2}) holds for all $f\in\mathcal{C}_p$, we deduce (from an argument analogous to that used to prove the fundamental lemma of the calculus of variations) that $p$ satisfies the differential equation (\ref{lolcofe11}).  This completes the proof.
 \end{proof}

 % \subsection{Application to obtaining formulas for densities}
 % \label{sec:appl-obta-form}
 We now show how the duality Lemma \ref{dulem} can be exploited to
 derive formulas for densities of distributions.  By duality, $p$
 satisfies the differential equation (\ref{lolcofe11}), and making the change of variables $y=\frac{b}{q^{n-m}}x^q$ yields the following differential equation
 \begin{equation}\label{difnear}T_{\frac{1-r_1}{q}}\cdots T_{\frac{1-r_n}{q}}p(y)-(-1)^{m+n}y T_{\frac{q+1-a_1}{q}}\cdots T_{\frac{q+1-a_m}{q}}p(y)=0.
 \end{equation}
 We recognise (\ref{difnear}) as an instance of the Meijer
 $G$-function differential equation (\ref{meidiffeqn}).  There are
 $\max\{m,n\}$ linearly independent solutions to (\ref{difnear}) that
 can be written in terms of the Meijer $G$-function (see \cite{olver},
 Chapter 16, Section 21).  Using a change of variables, we can thus
 obtain a fundamental system of solutions to (\ref{lolcofe11}) given as
 Meijer $G$-functions.  One can then arrive at a formula for the
 density by imposing the conditions that the solution must be
 non-negative and integrate to 1 over the support of the distribution.
 Due to the difficulty of handling the Meijer $G$-function, this final
 analysis is in general not straightforward.  However, one can
 ``guess'' a formula for the density based on the fundamental system
 of solutions, and then verify that this is indeed the density by an
 application of the Mellin transform (note that in this verification
 step there is no need to use the Mellin inversion formula).  An
 interesting direction for future research would be to develop
 techniques for identifying formulas for densities of distributions
 based solely on an analysis of the differential equation
 (\ref{lolcofe11}).  However, even as it stands, we have a technique for
 obtaining formulas for densities that may be intractable through
 standard methods.

 \vspace{2mm}
 
 \noindent{\emph{Products of $\mathrm{VG}(r,0,\sigma,0)$ random variables.}} Let $(Z_i)_{1\leq i\leq n}\sim\mathrm{VG}(r_i,0,\sigma_i,0)$ be independent, and set $Z=\prod_{i=1}^nZ_i$.
Recall the Stein operator (\ref{ngp21}) for the product of $\mathrm{VG}(r_i,0,\sigma_i,0)$ distributed random variables: 
 \begin{equation*}A_Zf(x)=\sigma^2 T_1^{n} T_{r_1} \ldots T_{r_n} -  M^2,
 \end{equation*}
 where $\sigma^2=\sigma_1^2 \ldots \sigma_n^2$.  By Lemma \ref{dulem}, it follows that the density $p$ satisfies the following differential equation:
\begin{equation}\label{vgode}T_0^nT_{1-r_1}\cdots T_{1-r_n}p(x)-\sigma^{-2}x^2p(x)=0.
\end{equation} 
Arguing as we did to obtain the ODE (\ref{difnear}), we make the substitution $y=\frac{x^2}{2^{2n}\sigma^2}$ to reduce (\ref{vgode}) to a G-function ODE of the type (\ref{meidiffeqn}).  We can therefore identify that the following function is a solution to (\ref{vgode}):
\begin{equation*}\label{vgdsoln}p(x)=CG^{2n,0}_{0,2n}\bigg(\frac{x^2}{2^{2n}\sigma^2}\; \bigg| \; \frac{r_1-1}{2},\ldots,\frac{r_n-1}{2},0,\ldots,0 \bigg),
\end{equation*}
where $C$ is an arbitrary constant.  We can apply (\ref{meijergintegration2}) to choose $C$ such the $p$ integrates to 1 across its support:
\begin{equation}\label{vgdsoln}p(x)=\frac{1}{2^n\pi^{n/2}\sigma}\prod_{j=1}^n\frac{1}{\Gamma(r_j/2)}G^{2n,0}_{0,2n}\bigg(\frac{x^2}{2^{2n}\sigma^2}\; \bigg| \; \frac{r_1-1}{2},\ldots,\frac{r_n-1}{2},0,\ldots,0 \bigg).
\end{equation}
We verify that this `guess' this is indeed the density using Mellin transforms; note that this verification is much more straightforward than an application of the Mellin inversion formula.

Let us define the Mellin transform and state some properties that will be useful to us.  The Mellin transform of a non-negative random variable $U$ with density $p$ is given by $M_U(s)=\E U^{s-1}$, for all $s$ such that the expectation exists.  If the random variable $U$ has density $p$ that is symmetric about the origin then we can define the Mellin transform of $U$ by
\begin{equation}\label{melu}M_U(s)=2\int_0^{\infty}x^{s-1}p(x)\,\mathrm{d}x.
\end{equation}
The Mellin transform is useful in determining the distribution of products of independent random variables due to the property that if the random variables $U$ and $V$ are independent then $M_{UV}(s)=M_U(s)M_V(s)$.

To obtain the Mellin transform of $Z=\prod_{i=1}^nZ_i$, we recall that $Z_i\stackrel{\mathcal{L}}{=}X_i\sqrt{Y_i}$, where $X_i\sim \mathcal{N}(0,\sigma_i^2)$ and $Y_i\sim \Gamma(r/2,1/2)$ are independent.  Using the formulas for the Mellin transforms of the normal and gamma distributions (see \cite{springer}), we have that
\begin{equation*}M_{X_i}(s)=\frac{1}{\sqrt{\pi}}2^{(s-1)/2}\sigma_i^{s-1}\Gamma(1/2), \quad M_{\sqrt{Y_i}}(s)= M_{Y_i}((s+1)/2)=2^{(s-1)/2}\frac{\Gamma(\frac{r_i-1+s}{2})}{\Gamma(r_i)},
\end{equation*}
and therefore
\begin{equation}\label{meleq}M_Z(s)=\frac{1}{\pi^{n/2}}2^{n(s-1)}\sigma^{s-1}[\Gamma(s/2)]^n\prod_{i=1}^n\frac{\Gamma(\frac{r_i-1+s}{2})}{\Gamma(\frac{r_i}{2})}.
\end{equation} 
Now, let $W$ denote a random variable with density (\ref{vgdsoln}).  Then, using (\ref{melu}) and (\ref{meijergintegration2}) gives that
\begin{align*}M_W(s)=2\times \frac{1}{2^n\pi^{n/2}\sigma}\prod_{j=1}^n\frac{1}{\Gamma(r_j/2)}\times \bigg(\frac{1}{2^{2n}\sigma^2}\bigg)^{-s/2}\times [\Gamma(s/2)]^n\times \prod_{i=1}^n\Gamma\bigg(\frac{r_i-1+s}{2}\bigg),
\end{align*}
which is equal to (\ref{meleq}).  Since the Mellin transforms of $W$ is equal to that of $Z$, it follows that $W$ is equal in law to $Z$.  Therefore (\ref{vgdsoln}) is indeed the p.d.f$.$ of the random variable $Z$.

\appendix

\section{List of Stein operators for continuous distributions}
\label{appendix:list}

Recall that $Mf(x) = x f(x)$, $I$ is the identity and $T_a f(x) = x f'(x) + a f(x)$.  We also recall that the beta function is defined by $B(a,b)=\frac{\Gamma(a)\Gamma(b)}{\Gamma(a+b)}$, and that $U(a,b,x)$ and $K_{\nu}(x)$ denote the confluent hypergeometric function of the second kind (\cite{olver}, Chapter 13) and the modified Bessel function of the second kind (\cite{olver}, Chapter 10), respectively.

We give a list of Stein operators for several classical probability distributions, in terms of the above operators.  References for these Stein operators are as follows: normal \cite{stein}, gamma \cite{diaconis, luk}, beta \cite{dobler beta, goldstein4, schoutens}, Student's $t$ \cite{schoutens}, inverse-gamma \cite{koudou}, $F$-distribution (new to this paper), PRR \cite{pekoz}, variance-gamma \cite{gaunt vg}, and generalized gamma \cite{gaunt ngb}. 
%Note the operators in Table \ref{tab:2} for the normal, PRR and variance-gamma distributions differ slightly from the usual operators from the literature, as they act on a different class of functions.  Bounds for solutions of a number of the Stein equations corresponding to the Stein operators of Table \ref{tab:2} are collected in \cite{dgv}.

The usual Stein operators (as defined in the above references) for the normal, PRR and variance-gamma distributions, are not in the form required in Section \ref{sec:general-results}. In these cases, we multiply the operators by $M$ on the right for the normal and variance-gamma distributions, and we multiply the operator by $M^2$ on the right for the PRR distribution. 
%(which is equivalent to applying them to $xf(x)$ instead of $f(x)$).
%It is important to note that by doing so, we change the class of functions the operators act on: if $\mathcal A$ acts on $\mathcal F$, then $\mathcal A M$ acts on $\{ f: Mf \in \mathcal F \}$ (in particular, if $\mathcal A f$ is defined when $f$ is smooth with compact support, so is $\mathcal A M$).

%We show in more detail the normal case.  The centered normal distribution with variance $\sigma^2$ has usual Stein operator given by $Af(x) = \sigma^2 f'(x) - x f(x)$, which reads, in our notation, $A = \sigma^2 D - M$. Applying this operator to $x f(x)$ instead of $f(x)$, or, equivalently, multiplying it on the right by $M$ leads to the new Stein operator $\tilde{ A} = \sigma^2 DM - M^2$. But $DM = MD+I = T_1$, so that $\tilde {A} = \sigma^2 T_1 - M^2$. This operator is indeed of the form of \eqref{eq:axlma}. The same trick is used for the PRR distribution and the variance-gamma distribution.

\begin{table}
\label{tab:1}
\begin{center}
\begin{tabular}{|c|c|c|}
\hline
Distribution & Parameters & Notation \\
\hline
 Normal & $\mu,\sigma \in \R$ & $\mathcal N(\mu,\sigma^2)$  \\
\hline
Gamma &$r,\lambda>0$ & $\Gamma(r,\lambda)$ \\
\hline
Beta & $a,b>0$ & $\mathrm{Beta}(a,b)$ \\
\hline
Student's $t$ &$\nu>0$ & $t_\nu $ \\
\hline
Inverse-gamma & $\alpha,\beta>0$ & $IG(\alpha,\beta)$ \\
\hline
$F$-distribution &$d_1,d_2>0$ & $F(d_1,d_2)$\\
\hline
PRR distribution & $s>1/2$ & $PRR_s$\\
\hline
Variance-gamma &  $r , \sigma> 0$, $\theta, \mu \in \R$ &$\mathrm{VG}(r,\theta,\sigma,\mu)$ \\
\hline
Generalized gamma & $r,\lambda,q>0$ & $\mathrm{GG}(r,\lambda,q)$\\
\hline
%K-distribution (1) & $ \mu>0, \nu> L>0$ &  $KD_1( \mu, \nu, L)$\\
%  \hline
%  K-distribution (2) & $ \lambda, c>0$ &  $KD_2( \lambda, c)$\\
%  \hline
%  Pareto & $\alpha,\beta>0$ & $\mathrm{Pareto}(\alpha,\beta)$ \\
%\hline
\end{tabular}
\end{center}
\caption{Distributions}
\end{table}

\begin{table}
\label{tab:2}
\begin{center}
\begin{tabular}{|c|c|c|}
\hline
Distribution & p.d.f$.$ & Stein operator\\
\hline
  $\mathcal N(\mu,\sigma^2)$  & $\frac{1}{\sqrt{2\pi}\sigma}\mathrm{e}^{-(x-\mu)/\sigma^2}$ & $\sigma^2 T_1 +\mu M- M^2$ \\
\hline
$\Gamma(r,\lambda)$& $\frac{\lambda^r}{\Gamma(r)}x^{r-1}\mathrm{e}^{-\lambda x} \indic{x>0}$ &  $T_r-\lambda M$.\\
\hline
 $\mathrm{Beta}(a,b)$& $\frac{1}{B(a,b)}x^{a-1}(1-x)^{b-1}\indic{0<x<1}$ & $T_a-MT_{a+b}$\\
\hline
 $t_\nu $&  $\frac{\Gamma(\frac{\nu+1}{2})}{\sqrt{\nu\pi}\Gamma(\frac{\nu}{2})}\big(1+\frac{x^2}{\nu}\big)^{-(\nu+1)/2}$  & $\nu T_1+M^2T_{2-\nu}$\\
\hline
$IG(\alpha,\beta)$& $\frac{\beta^\alpha}{\Gamma(\alpha)}x^{-\alpha-1}\mathrm{e}^{-\beta /x}\, \indic{x>0}$ & $\beta I+MT_{1-\alpha}$\\
\hline
$F(d_1,d_2)$&$\frac{1}{B(\frac{d_1}{2},\frac{d_2}{2})}\big(\frac{d_1}{d_2}\big)^{d_1/2}x^{d_1/2-1}\big(1+\frac{d_1}{d_2}x\big)^{-(d_1+2_2)/2} \; \indic{x>0}$ & $d_2T_{d_1/2}+d_1MT_{1-d_2/2}$\\
\hline
$PRR_s$ & $\Gamma(s)\sqrt{\frac{2}{s\pi}}\exp\big(-\frac{x^2}{2s}\big)U\big(s-1,\frac{1}{2},\frac{x^2}{2s}\big) \indic{x>0}$ & $sT_1T_2-M^2T_{2s}$\\
\hline
$\mathrm{VG}(r,\theta,\sigma, \mu = 0)$ &  $\frac{1}{\sigma\sqrt{\pi} \Gamma(\frac{r}{2})} \mathrm{e}^{\frac{\theta}{\sigma^2} x} \big(\frac{|x|}{2\sqrt{\theta^2 +  \sigma^2}}\big)^{\frac{r-1}{2}} K_{\frac{r-1}{2}}\big(\frac{\sqrt{\theta^2 + \sigma^2}}{\sigma^2} |x| \big)$ & $\sigma^2T_1T_r+2\theta MT_{r/2}-M^2$\\
\hline
 $\mathrm{GG}(r,\lambda,q)$ & $\frac{q \lambda^r}{ \Gamma(r/q)} x^{r-1}\mathrm{e}^{-(\lambda x)^q} \, \indic{x>0}$ & $T_r - q\lambda^q M^q$ \\
 \hline
%  $KD_1( \mu, \nu, L)$ & $\frac{2 \left( \sqrt{\frac{L \nu}{\mu}}
%    \right)^{L+\nu}\sqrt{x}^{L+ \nu- 2}}{ \Gamma(L)\Gamma(\nu)}
%  K_{\nu-L}\left( 2 \sqrt{\frac{L \nu}{\mu}} \sqrt{x} \right)$  & $\frac{\mu}{L \nu}    T_L T_{\nu} - M$ \\
% \hline
%  $KD_2( \lambda, c)$ &  $\frac{2
%    c}{\Gamma(\lambda)} \left( \frac{cx}{2}
%  \right)^{\lambda} K_{\lambda-1}(cx) $ & $\frac{\mu}{\lambda}  T_2 T_{2\lambda} - 2  M^2$ \\
% \hline
%$\mathrm{Pareto}(\alpha,\beta)$ & $\frac{\alpha\beta^\alpha}{(x+\beta)^{\alpha+1}} \, \indic{x\geq0}$ & $\beta T_1+MT_{1-\alpha}$ \\
%\hline
\end{tabular}
\end{center}
\caption{p.d.f$.$ and Stein operator of some classical
  distributions.}
\end{table}

\section{The Meijer $G$-function}
\label{sec:meijerg}

%Here we define the  Meijer $G$-function and present some of its basic properties that are relevant to this paper.  For further properties of this function see \cite{luke, olver}.

The Meijer $G$-function is defined (see see \cite{luke, olver}), for $z\in\C\setminus\{0\}$, by the contour integral:
\[G^{m,n}_{p,q}\bigg(z \; \bigg|\; {a_1,\ldots, a_p \atop b_1,\ldots,b_q} \bigg)=\frac{1}{2\pi i}\int_{c-i\infty}^{c+i\infty}z^{-s}\frac{\prod_{j=1}^m\Gamma(s+b_j)\prod_{j=1}^n\Gamma(1-a_j-s)}{\prod_{j=n+1}^p\Gamma(s+a_j)\prod_{j=m+1}^q\Gamma(1-b_j-s)}\,\mathrm{d}s,\]
where $c$ is a real constant defining a Bromwich path separating the poles of $\Gamma(s + b_j)$ from those of $\Gamma(1- a_j- s)$ and where we use the convention that the empty product is $1$.

The following formula follows from \cite{luke}, formula (1) of Section 5.6 and a change of variables:
\begin{equation}\label{meijergintegration2}\int_0^{\infty}x^{s-1}G_{p,q}^{m,n}\bigg(\alpha x^\gamma \; \bigg| \;{a_1,\ldots,a_p \atop b_1,\ldots,b_q}\bigg)\,\mathrm{d}x=\frac{\alpha^{-s/\gamma}}{\gamma}\frac{\prod_{j=1}^m\Gamma(b_j+\frac{s}{\gamma})\prod_{j=1}^n\Gamma(1-a_j-\frac{s}{\gamma})}{\prod_{j=m+1}^q\Gamma(1-b_j-\frac{s}{\gamma})\prod_{j=n+1}^p\Gamma(a_j+\frac{s}{\gamma})}.
\end{equation}
For the conditions under which this formula is valid see \cite{luke}, pp$.$ 158--159.  In particular, the formula is valid when $n=0$, $1\leq p+1\leq m\leq q$ and $\alpha>0$.

The $G$-function $f(z)=G^{m,n}_{p,q}\big(z\big|{a_1,\ldots,a_{p} \atop b_1,\ldots,b_q}\big)$ satisfies the differential equation
\begin{equation}\label{meidiffeqn}(-1)^{p-m-n}zT_{1-a_1}\cdots T_{1-a_p}f(z)-T_{-b_1}\cdots T_{-b_q}f(z)=0.
\end{equation}

\section*{Acknowledgements} The authors would like to thank the
referees for their constructive comments, particularly to one referee
for their careful reading of our paper and their suggestions which lead to a substantial improvement in the
organisation of the paper. RG acknowledges support from EPSRC grant
EP/K032402/1 and is currently supported by a Dame Kathleen Ollerenshaw
Research Fellowship.  RG is grateful to Universit\'e de Li\`ege, FNRS
and EPSRC for funding a visit to University de Li\`ege, where some of
the details of this project were worked out.  YS gratefully
acknowledges support by the Fonds de la Recherche Scientifique - FNRS
under Grant MIS F.4539.16. Part of GM's research was supported by a WG
(Welcome Grant) from Universit\'e de Li\`ege.

\end{document}